\documentclass[11]{amsart}

\usepackage{amsmath, amstext, amsgen, amsbsy, amsopn, amsfonts, amssymb, graphicx,psfrag}
\usepackage{amsthm}
\usepackage{comment}
\usepackage[normalem]{ulem}
\usepackage{pdfsync}
\usepackage{epstopdf}
\usepackage{color}
\usepackage[thinlines]{easytable}
\usepackage{multirow}
\usepackage{multicol}
\usepackage{hyperref}
\usepackage{enumitem}

\def\Z{\mathbb{Z}}
\def\e{{\varepsilon}}
\newcommand{\abs}[1]{\vert #1 \vert}
\newcommand{\inv}[1]{\overline{#1}\,}

\newtheorem{theorem}{Theorem}[section]
\newtheorem{lemma}[theorem]{Lemma}
\newtheorem{corollary}[theorem]{Corollary}
\theoremstyle{definition}
\newtheorem*{definition}{Definition}
\newtheorem*{question}{Question}

\theoremstyle{remark}
\newtheorem{remark}{Remark}[section]
\begin{document}

\markboth{Backer Peral and Mellor}{$N$-quandles of spatial graphs}

\title{$N$-quandles of spatial graphs}

\author{Ver\'{o}nica Backer Peral and Blake Mellor}
\thanks{This paper includes results from the first author's senior thesis at LMU}\ \
\address{Loyola Marymount University, 1 LMU Drive, Los Angeles, CA 90045}
\email{blake.mellor@lmu.edu}

\date{}
\maketitle
\begin{abstract} 
The fundamental quandle is a powerful invariant of knots, links and spatial graphs, but it is often difficult to determine whether two quandles are isomorphic. One approach is to look at quotients of the quandle, such as the $n$-quandle defined by Joyce \cite{JO}; in particular, Hoste and Shanahan \cite{HS2} classified the knots and links with finite $n$-quandles. Mellor and Smith \cite{MS} introduced the $N$-quandle of a link as a generalization of Joyce's $n$-quandle, and proposed a classification of the links with finite $N$-quandles. We generalize the $N$-quandle to spatial graphs, and investigate which spatial graphs have finite $N$-quandles. We prove basic results about $N$-quandles for spatial graphs, and conjecture a classification of spatial graphs with finite $N$-quandles, extending the conjecture for links in \cite{MS}.  We verify the conjecture in several cases, and also present a possible counterexample.
\end{abstract}



\section{Introduction}\label{S:intro}

The {\it fundamental quandle} of a knot or link was introduced by Joyce \cite{JO2, JO} and, independently, by Matveev \cite{Ma}. The fundamental quandle is a complete invariant of tame knots (up to a change of orientation); unfortunately, classifying quandles is not much easier than classifying knots. One approach is to look at quotients of the fundamental quandle; of particular interest are cases when the quotients are finite, and so may be relatively easily computed and compared.

Joyce \cite{JO2, JO} introduced the $n$-quandle, where every element of the quandle has a finite ``order'' of $n$. Hoste and Shanahan \cite{HS2} proved that for a link $L$ the $n$-quandle $Q_n(L)$ is finite if and only if $L$ is the singular locus (with each component labeled $n$) of a spherical 3-orbifold with underlying space $\mathbb{S}^3$. This result, together with Dunbar's \cite{DU} classification of all geometric, non-hyperbolic 3-orbifolds, allowed them to give a complete list of all  knots and links in $\mathbb{S}^3$ with finite $n$-quandles for some $n$ \cite{HS2}.  Many of these finite $n$-quandles have been described in detail \cite{CHMS, HS1, Me}.

Some orbifolds in Dunbar's paper have a singular locus that is a link with different labels on different components. With this motivation, Mellor and Smith \cite{MS} defined $N$-quandles as a generalization of $n$-quandles, where now elements in different components of the quandle have different ``orders''. They proved that every labeled link appearing as the singular locus of a spherical orbifold with underlying space $\mathbb{S}^3$ in Dunbar's classification has a corresponding finite $N$-quandle, and conjectured that these are the only links with finite $N$-quandles.

However, Dunbar's classification also includes orbifolds whose singular locus is a graph with labels on the edges. Niebrzydowski \cite{Ni} defined fundamental quandles for spatial graphs, and the notion of the $n$-quandle and $N$-quandle are easily extended to this context. So it is natural to again conjecture that a spatial graph has a finite $N$-quandle if and only if it appears in Dunbar's classification. The purpose of this paper is to put forward this conjecture, and to investigate the evidence both for and against it.  In particular, we show that many of the graphs in Dunbar's list do, indeed, have finite $N$-quandles, but also identify a potential counterexample.

In section \ref{S:quandles}, we will review the definitions of quandles and $N$-quandles, and of the fundamental quandle (and $N$-quandle) of a link or spatial graph. We also prove some elementary results about $N$-quandles of spatial graphs. At the end of this section we state our Main Conjecture:\medskip

\noindent {\bf Main Conjecture.} {\em A spatial graph $G$ has a finite $N$-quandle if and only if there is a spherical orbifold with underlying space $\mathbb{S}^3$ whose singular locus is the edge-labeled spatial graph $(H,M)$, where $(G,N)$ divides a graph $(G, N')$ that is homeomorphic to a subgraph of $(H,M)$.} \medskip

Our primary approach to verifying this conjecture for particular spatial graphs is to compute the {\em Cayley graphs} of the associated $N$-quandles. We review the algorithm to compute the Cayley graph of a quandle in section \ref{S:Cayley}. In section \ref{S:exceptional} we will consider specific labeled graphs which appear in Dunbar's classification of of 3-orbifolds; we show that all but one of them (our potential counterexample) has a finite $N$-quandle. In section \ref{S:families} we verify the conjecture for some infinite families of graphs by explicitly computing the size of their $N$-quandles (see Theorem \ref{T:Gkmn} and Corollary \ref{C:Gkm}).  Finally, we will pose some questions for further investigation.

\section{Quandles and spatial graphs} \label{S:quandles}

\subsection{Quandles, $n$-quandles and $N$-quandles} \label{SS:quandles}

We begin with a review of the definition of a quandle and its associated $n$-quandles. We refer the reader to \cite{FR}, \cite{JO2}, \cite{JO}, and \cite{WI} for more detailed information.

A {\it  quandle} is a set $Q$ equipped with two binary operations $\rhd$ and $\rhd^{-1}$ that satisfy the following three axioms:
\begin{itemize}
\item[\bf A1.] $x \rhd x =x$ for all $x \in Q$.
\item[\bf A2.] $(x \rhd y) \rhd^{-1} y = x = (x \rhd^{-1} y) \rhd y$ for all $x, y \in Q$.
\item[\bf A3.] $(x \rhd y) \rhd z = (x \rhd z) \rhd (y \rhd z)$ for all $x,y,z \in Q$.
\end{itemize}

Each element $x\in Q$ defines a map $S_x:Q \to Q$ by $S_x(y)=y \rhd x$. The axiom A2 implies that each $S_x$ is a bijection and the axiom A3 implies that each $S_x$ is a quandle homomorphism, and therefore an automorphism. We call $S_x$ the {\it point symmetry at $x$}. The {\em inner automorphism group} of $Q$, Inn$(Q)$, is the group of automorphisms generated by the point symmetries.

It is important to note that the operation $\rhd$ is, in general, not associative. To clarify the distinction between $(x \rhd y) \rhd z$ and $x \rhd (y \rhd z)$, we adopt the exponential notation introduced by Fenn and Rourke in \cite{FR} and denote $x \rhd y$ as $x^y$ and $x \rhd^{-1} y$ as $x^{\bar y}$. With this notation, $x^{yz}$ will be taken to mean $(x^y)^z=(x \rhd y)\rhd z$ whereas $x^{y^z}$ will mean $x\rhd (y \rhd z)$. 

The following useful lemma from \cite{FR} describes how to re-associate a product in a quandle given by a presentation. 

\begin{lemma} \label{leftassoc}
If $a^u$ and $b^v$ are elements of a quandle, then
$$\left(a^u \right)^{\left(b^v \right)}=a^{u \bar v b v} \ \ \ \ \mbox{and}\ \ \ \ \left(a^u \right)^{\overline{\left(b^v \right)}}=a^{u \bar v \bar b v}.$$
\end{lemma}

Using Lemma~\ref{leftassoc}, elements in a quandle given by a presentation $\langle S \mid R \rangle$ can be represented as equivalence classes of expressions of the form $a^w$ where $a$ is a generator in $S$ and $w$ is a word in the free group on $S$ (with $\bar x$ representing the inverse of $x$).

If $n$ is a natural number, a quandle $Q$ is an {\em $n$-quandle} if $x^{y^n} =x$ for all $x,y \in Q$, where by $y^n$ we mean $y$ repeated $n$ times. Given a presentation $\langle S \,|\, R\rangle$ of $Q$, a presentation of  the quotient $n$-quandle $Q_n$ is obtained by adding the relations $x^{y^n}=x$ for every pair of distinct generators $x$ and $y$. 

The action of the inner automorphism group Inn$(Q)$ on the quandle $Q$ decomposes the quandle into disjoint orbits. These orbits are the {\em components} (or {\em algebraic components}) of the quandle $Q$; a quandle is {\em connected} if it has only one component. We generalize the notion of an $n$-quandle by picking a different $n$ for each component of the quandle.

\begin{definition} \label{D:Nquandle}
Given a quandle $Q$ with $k$ ordered components, labeled from 1 to $k$, and a $k$-tuple of natural numbers $N = (n_1, \dots, n_k)$, we say $Q$ is an {\em $N$-quandle} if $x^{y^{n_i}} = x$ whenever $x \in Q$ and $y$ is in the $i$th component of $Q$. 
\end{definition}

Note that the ordering of the components in an $N$-quandle is very important; the relations depend intrinsically on knowing which component is associated with which number $n_i$. 

Given a presentation $\langle S \,|\, R\rangle$ of $Q$, a presentation of  the quotient $N$-quandle $Q_N$ is obtained by adding the relations $x^{y^{n_i}}=x$ for every pair of distinct generators $x$ and $y$, where $y$ is in the $i$th component of $Q$. An $n$-quandle is then the special case of an $N$-quandle where $n_i = n$ for every $i$. 

\subsection{Fundamental quandles of links and spatial graphs} \label{SS: fundamental}

If $G$ is an oriented knot, link or spatial graph in $\mathbb{S}^3$, then a presentation of its fundamental quandle, $Q(G)$, can be derived from a regular diagram $D$ of $G$ by a process similar to the Wirtinger algorithm (this was described for links by Joyce \cite{JO}, and extended to spatial graphs by Niebrzydowski \cite{Ni}). We assign a quandle generator $x_1, x_2, \dots , x_n$ to each arc of $D$, then introduce relations at each crossing and (for spatial graphs) vertex. At a crossing, we introduce the relation $x_i=x_k^{ x_j}$ as shown on the left in Figure~\ref{relations}. At a vertex with incident edges $a_1, a_2, \dots a_n$, as shown on the right in Figure~\ref{relations}, we introduce the relation $((x\rhd^{\e_1} a_1)\rhd^{\e_2} a_2) \cdots \rhd^{\e_n} a_n = x$ (where $\e_i = 1$ if $a_i$ is directed into the vertex, and $\e_i = -1$ if $a_i$ is directed out from the vertex). It is easy to check that the Reidemeister moves for spatial graphs do not change the quandle given by this presentation so that the quandle is indeed an invariant of the oriented spatial graph (or link).

\begin{figure}[h]
$$\includegraphics[height=1in]{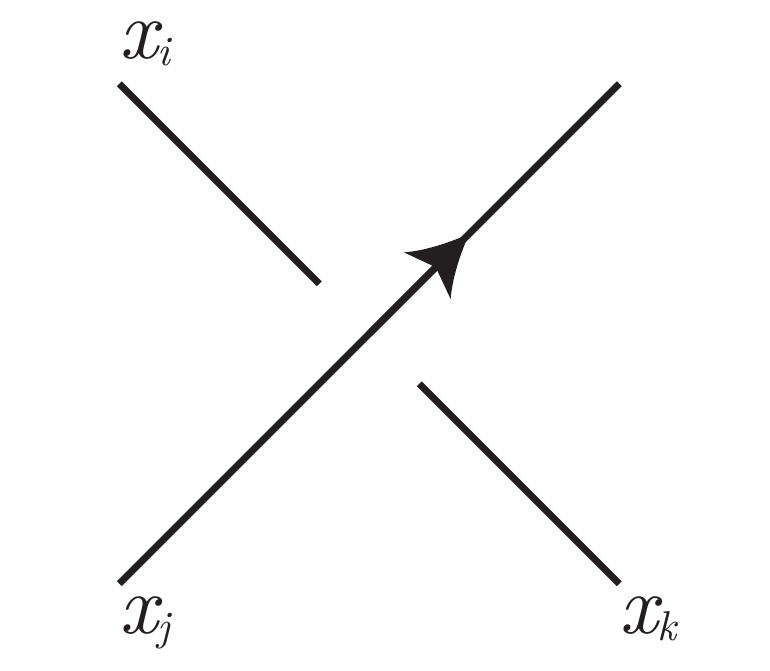} \qquad \qquad \includegraphics[height = 1in]{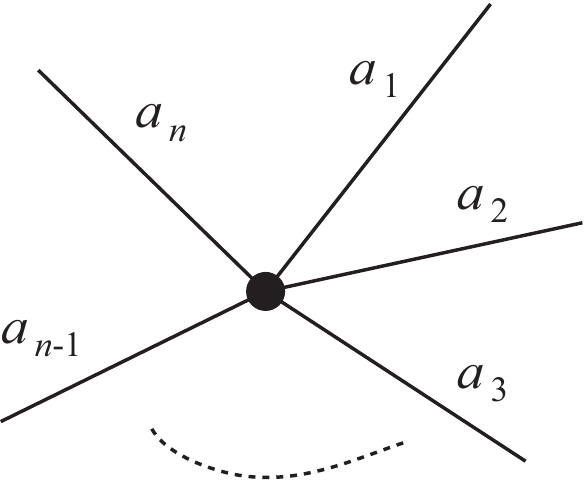}$$
$$\qquad\qquad \qquad x_i=x_k^{x_j} \qquad \qquad ((x\rhd^{\e_1} a_1)\rhd^{\e_2} a_2) \cdots \rhd^{\e_n} a_n = x$$
\caption{The fundamental quandle relations at a crossing and at a vertex.}
\label{relations}
\end{figure}

If $n$ is a natural number, we can take the quotient $Q_n(G)$ of the fundamental quandle $Q(G)$ to obtain the fundamental $n$-quandle of a spatial graph. Hoste and Shanahan \cite{HS2} classified all pairs $(L,n)$ for which $Q_n(L)$ is finite, where $L$ is a link.

Fenn and Rourke \cite{FR} observed that for a link $L$, the components of the quandle $Q(L)$ are in bijective correspondence with the components of the link $L$, with each component of the quandle containing the generators of the Wirtinger presentation associated to the corresponding link component. Similarly, for a spatial graph $G$, the components of the quandle $Q(G)$ correspond to the {\bf edges} of the graph $G$. This is because two distinct generators of the quandle (from the Wirtinger presentation) are in the same component if and only if the corresponding arcs of the diagram are separated by a sequence of crossings; hence they must lie on the same edge.

So if we have a graph $G$ with $k$ edges, and label each edge $e_i$ with a natural number $n_i$, we can let $N = (n_1, \dots, n_k)$ and take the quotient $Q_N(G)$ of the fundamental quandle $Q(G)$ to obtain the fundamental $N$-quandle of the graph (this depends on the ordering of the edges). If $Q(G)$ has the Wirtinger presentation from a diagram $D$, then we obtain a presentation for $Q_N(G)$ by adding relations $x^{y^{n_i}} = x$ for each pair of distinct generators $x$ and $y$ where $y$ corresponds to an arc of edge $e_i$ in the diagram $D$.

\begin{remark} \label{R:Nquandle}
It is worth observing that if $x_i^y = x_i$ for every {\em generator} $x_i$ of a quandle, then $x^y = x$ for every {\em element} $x$ of the quandle.  Say $x = x_1^{x_2x_3\dots x_m}$, where each $x_i$ is a generator.  Then
$$x^y = x_1^{x_2x_3\cdots x_m y} = x_1^{y(\bar{y}x_2 y)(\bar{y} x_3 y) \cdots (\bar{y} x_m y)} = (x_1^y)^{(x_2^y)(x_3^y) \cdots (x_m^y)} = x_1^{x_2x_3\dots x_m} = x.$$
We will use this fact when constructing Cayley graphs for $N$-quandles.
\end{remark}

\subsection{Properties of $N$-quandles} \label{SS:properties}

In this section, we will make some observations about $N$-quandles, particularly for spatial graphs. Given two $k$-tuples $N = (n_1, \dots, n_k)$ and $M = (m_1, \dots, m_k)$, we say that $M$ {\em divides} $N$ (or $M \vert N$) if $m_i \vert n_i$ for each $i$. If $G$ is a spatial graph with $k$ edges, we will also say the labeled graph $(G, M)$ divides the labeled graph $(G, N)$.

\begin{lemma} \label{L:divide}
If $G$ is a spatial graph with $k$ edges (or a link with $k$ components), and $N$ and $M$ are $k$-tuples with $M \vert N$, then $\vert Q_M(G)\vert \leq \vert Q_N(G) \vert$. In particular, if $Q_N(G)$ is finite, so is $Q_M(G)$.
\end{lemma}

\begin{proof}
Since the graph is the same, $Q_M(G)$ and $Q_N(G)$ have the same crossing and vertex relations, and the same number of components. The only difference is that, if $y$ is in the $i$th component, then in $Q_M(G)$ we have $x^{y^{m_i}} = x$ (for any element $x$), and in $Q_N(G)$ we have $x^{y^{n_i}} = x$.  Since $m_i \vert n_i$, this means the relation $x^{y^{n_i}} = x$ holds in both quandles.  So every relation in $Q_N(G)$ holds in $Q_M(G)$, which means $Q_M(G)$ is a quotient of $Q_N(G)$, and hence smaller (or the same cardinality).
\end{proof}

\begin{lemma} \label{L:delete}
Let $G$ be a spatial graph with $k$ edges $e_1, \dots, e_k$ (or a link with $k$ components).  Let $G_i = G - e_i$ (the result of deleting edge (or component) $e_i$). If $N = (n_1, \dots, n_k)$, let $N_i = (n_1, \dots, n_{i-1}, n_{i+1}, \dots, n_k)$. Also let $C_i$ be the component of $Q_N(G)$ corresponding to edge $e_i$. Then $\vert Q_{N_i}(G_i) \vert \leq \vert Q_N(G) - C_i\vert$.  So if $Q_N(G)$ is finite, so is $Q_{N_i}(G_i)$. In particular, if $n_i = 1$, then $\vert Q_{N_i}(G_i) \vert = \vert Q_N(G) - C_i\vert$.
\end{lemma}

\begin{proof} 
To obtain $Q_{N_i}(G_i)$ from $Q_N(G)$, you simply remove the component $C_i$, and then add the relations $x^y = x$ for all generators $x$ and all generators $y$ corresponding to arcs along $e_i$. Since we are adding relations, the quandle cannot get any larger, so $\vert Q_{N_i}(G_i) \vert \leq \vert Q_N(G) - C_i\vert$. In the case when $n_i = 1$, the relations $x^y = x$ were already present, so the only change is removing the component $C_i$.
\end{proof}

\begin{lemma} \label{L:vertex}
Consider a graph $G$ with edges $e_1, \dots, e_k$ and an edge labeling $N = (n_1, \dots, n_k)$. Let $G^i$ be the result of adding a vertex $v$ of degree 2 to $e_i$, splitting it into edges $f$ and $g$, and give both $f$ and $g$ the same label as $e_i$.  So $G^i$ has edge labeling $N^i = (n_1, \dots, n_i, n_i, \dots, n_k)$. Let $C_i$ be the component of $Q_N(G)$ corresponding to edge $e_i$, and let $C_i'$ be an isomorphic copy of $C_i$.  Then $Q_{N^i}(G^i) = Q_N(G) \cup C_i'$. In particular, if one of $Q_{N^i}(G^i)$ or $Q_N(G)$ is finite, so is the other.
\end{lemma}

\begin{proof}
Suppose $v$ is added to arc $a$, and splits it into arcs $b$ and $c$, with orientations induced by the orientation on $a$.  The vertex relation at $v$ is $x^{b\bar{c}} = x$, or $x^b = x^c$, where $x$ is any element of the quandle. Any relation of $Q_N(G)$ has a corresponding relation in $Q_{N^i}(G^i)$, with any occurrence of $a$ replaced by $b$ or $c$.  But since $x^b = x^c$, we may assume $a$ is simply replaced by $b$ everywhere. So then $Q_N(G)$ and $Q_{N^i}(G^i)$ have exactly the same quandle relations (up to replacing $a$ by $b$); the only difference is that $Q_{N^i}(G^i)$ has an extra component corresponding to the extra edge.  However, since the vertex $v$ may be placed anywhere along edge $e_i$ without changing the graph topologically, the components $C_i$ and $C_i'$ corresponding to the edges $f$ and $g$ can be exchanged by an automorphism of the graph.  Hence, these components must be isomorphic, completing the proof.
\end{proof}

We will say that edge-labeled spatial graphs $(G, N)$ and $(H, M)$ are {\em homeomorphic} if one can be obtained from the other by adding and/or removing vertices of degree 2, modifying the labelings at each step as in Lemma \ref{L:vertex}. With these observations, we can state our main conjecture. \medskip

\noindent {\bf Main Conjecture.} {\em A spatial graph $G$ has a finite $N$-quandle if and only if there is a spherical orbifold with underlying space $\mathbb{S}^3$ whose singular locus is the edge-labeled spatial graph $(H,M)$, where $(G,N)$ divides a graph $(G, N')$ that is homeomorphic to a subgraph of $(H,M)$.} \medskip

\section{Computing Cayley graphs} \label{S:Cayley}

Given a presentation of a quandle, one can try to systematically enumerate its elements and simultaneously produce a Cayley graph of the quandle. This is our primary means of proving that a quandle is finite. Such a method was described in a graph-theoretic fashion by Winker in \cite{WI}. The method is similar to the well-known Todd-Coxeter process for enumerating cosets of a subgroup of a group \cite{TC} and has been extended to racks by Hoste and Shanahan \cite{HS3}. (A rack is more general than a quandle, requiring only axioms A2 and A3.) We provide a brief description of Winker's method applied to the $N$-quandle of a spatial graph (or link).  Suppose $G$ is a labeled spatial graph diagram with $c$ crossings and $v$ vertices, and $Q_N(G)$ is presented as
$$Q_N(G)=\left\langle x_1, x_2, \dots, x_g \, \left\vert \, \left\{x_{j_i}^{w_i}=x_{k_i} \right\}_{i = 1}^c, \left\{x^{u_i} = x\right\}_{i = 1}^v, \left\{x^{x_i^{n_i}} = x\right\}_{i = 1}^g \right.\right\rangle,$$
where each $w_i$ and $u_i$ is a word in $\{x_1,\dots , x_g, \overline{x_i}, \dots, \overline{x_g}\}$ (representing the crossing and vertex relations, respectively), and $n_i$ is the label on the quandle component containing $x_i$. As noted in Remark \ref{R:Nquandle}, for any word $w$, we use $x^w = x$ as shorthand for the set of relations $\left\{x_i^{w}= x_i\right\}_{i = 1}^g$; $x$ may then be understood to be any element of the quandle. 

If $y$ is any element of the quandle, then it follows from the relation $x_{j_i}^{w_i}=x_{k_i}$ and Lemma~\ref{leftassoc} that $y^{\overline{w}_i x_{j_i}w_i}=y^{x_{k_i}}$, and so
$$y^{\overline{w}_i x_{j_i} w_i \overline{x}_{k_i}}=y \text{ for all } y \text{ in } Q_N(G).$$

Winker calls this relation the {\it secondary relation}  associated to the {\it primary relation} $x_{j_i}^{w_i}=x_{k_i}$. We also consider relations of the form $y^{u_i} = y$, for $1 \le i \le v$ and $y^{ x_i^{n_i} }=y$ for all $y$ and $1 \le i \le g$ to be secondary relations (since they apply to all elements of the quandle).

Winker's method now proceeds to build the Cayley graph associated to the presentation as follows:
  
\begin{enumerate}
\item Begin with $g$ vertices labeled $x_1,x_2, \dots, x_g$ and numbered $1,2, \dots,g$. 
\item Add an oriented loop at each vertex $x_i$ and label it $x_i$. (This encodes the axiom A1.)
\item For each value of $i$ from $1$ to $r$, {\em trace} the primary relation $x_{j_i}^{w_i}=x_{k_i}$ by introducing new vertices and oriented edges as necessary to create an oriented path from $x_{j_i}$ to $x_{k_i}$ given by $w_i$. Consecutively number (starting from $g+1$) new vertices in the order they are introduced.  Edges are labelled with their corresponding generator and oriented to indicate whether $x_i$ or $\overline x_i$ was traversed. 
\item Tracing a relation may introduce edges with the same label and same orientation into or out of a shared vertex. We identify all such edges, possibly leading to other identifications. This process is called {\it collapsing} and all collapsing is carried out before tracing the next relation. 
\item Proceeding in order through the vertices, trace and collapse each secondary relation (in order). All of these relations are traced and collapsed at a vertex before proceeding to the next vertex.
\end{enumerate}

The method will terminate in a finite graph if and only if the $N$-quandle is finite. The reader is referred to Winker~\cite{WI} and Hoste and Shanahan \cite{HS3}  for more details. Code implementing the algorithm in {\em Mathematica} and {\em Python} is available on the author's webpage \cite{Me2}, and was used to do the calculations in section \ref{S:exceptional}.

\section{Exceptional graphs} \label{S:exceptional}

Dunbar \cite{DU} classifies 3-dimensional orbifolds into several types.  The spherical orbifolds are either of type 2, meaning that they are Seifert fibered orbifolds with a 2-orbifold base, or type 4, meaning they do not fiber over 2-orbifolds. There are several infinite families of spherical orbifolds of type 2, but only 18 of type 4, all of which have a graph (rather than a link) as the singular locus. In this section, we will consider these 18 exceptional (labeled) graphs. We will also include other labelings on these graphs that divide the ones given in Dunbar (though these fall into some of families of type 2 orbifolds). The results in this section were found by directly computing the Cayley graphs of the relevant quandles, as described in Section \ref{S:Cayley}.

\begin{figure}[htbp]
\begin{center}
{$
\begin{array}{ccc}
\includegraphics[height=1in,trim=0 0 0 0,clip]{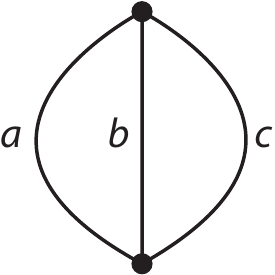} & & \includegraphics[height=1in,trim=0 0 0 0,clip]{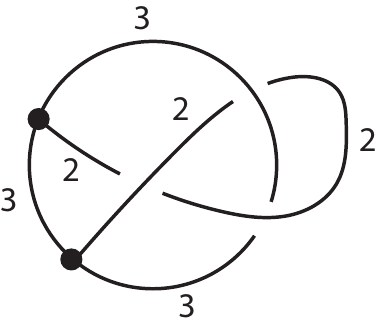} \\
\text{Theta graph } \theta_3 & & \text{Knotted theta graph } KT \\
&& \\
\end{array}
$}
{$
\begin{array}{ccc}
\includegraphics[height=.8in,trim=0 0 0 0,clip]{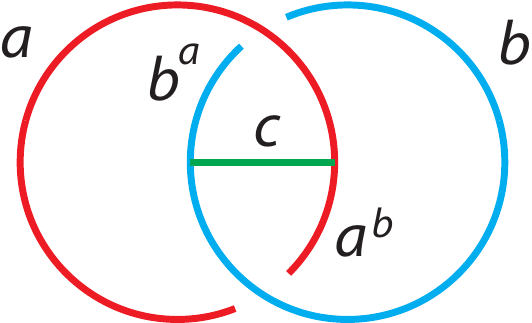} & \includegraphics[height=1in,trim=0 0 0 0,clip]{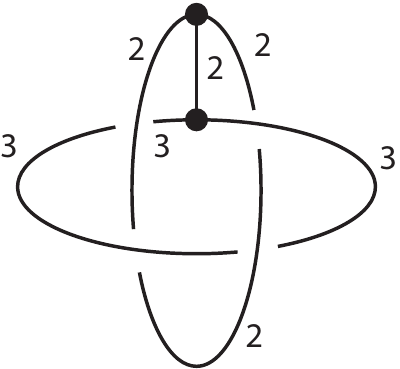} & \includegraphics[height=.8in,trim=0 0 0 0,clip]{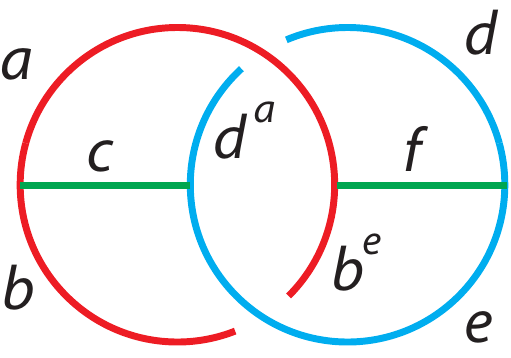} \\
\text{Hopf Handcuff graph } H_1 & \text{2-linked Handcuff graph } H_2 & \text{Double Handcuff graph } DH \\
&& \\
\end{array}
$}
{$
\begin{array}{ccc}
\includegraphics[height=1in,trim=0 0 0 0,clip]{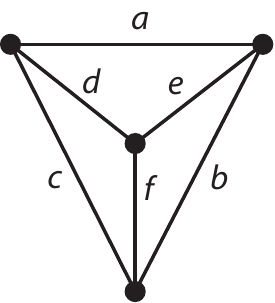} & & \includegraphics[height=1in,trim=0 0 0 0,clip]{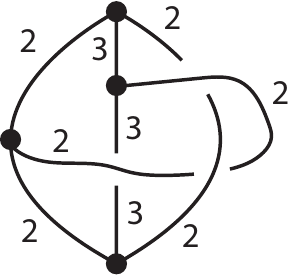} \\
\text{Planar } K_4 & & \text{Knotted } K_4 
\end{array}
$}
\end{center}
\caption{The exceptional graphs.}
\label{F:exceptional}
\end{figure}

Figure~\ref{F:exceptional} shows the exceptional graphs, with the edges labeled in alphabetical order, and with a presentation for the fundamental quandle (given the choice of orientations shown).  To simplify the presentations, we have reduced them to just use one generator for each edge, so we do not need to include the crossing relations.  Moreover, in some cases, the vertex relations are redundant, so there are fewer relations than vertices.  Table~\ref{T:exceptional} lists all the labelings of these graphs shown in Dunbar (the order of the labels corresponds to the alphabetical order of the edges in Figure \ref{F:exceptional}), and the size of the corresponding $N$-quandle. Labelings that do not correspond to an orbifold of type 4 (i.e. not shown in Table 8 of \cite{DU}) are marked with an asterisk.

\begin{table}[tbp]
\begin{center}
\begin{tabular}{|c|c|c|}
\hline
Graph $G$ & $N$ & $\vert Q_N(G)\vert$ \\ \hline
\multirow{5}{*}{$\theta_3$} & (2,2,2)* & 6 \\ \cline{2-3}
& (3,2,2)* & 8 \\ \cline{2-3}
& (3,3,2) & 14 \\ \cline{2-3}
& (4,3,2) & 26 \\ \cline{2-3}
& (5,3,2) & 62 \\ \hline
$KT$ & (3,3,2) & 1680 \\ \hline
\multirow{2}{*}{$H_1$} & (3,2,2) & 32 \\ \cline{2-3}
& (3,3,2) & 336 \\ \hline
$H_2$ & (3,2,2) & 768 \\ \hline
\multirow{3}{*}{$DH$} & (2,2,2,3,2,2)* & 102 \\ \cline{2-3}
& (2,2,3,3,2,2) & 320 \\ \cline{2-3}
& (2,2,2,3,2,4) & 2976 \\ \hline
\multirow{9}{*}{Planar $K_4$} & (3,2,2,2,2,2)* & 34 \\ \cline{2-3}
& (3,3,2,2,2,2) & 64 \\ \cline{2-3} 
& (3,4,2,2,2,2) & 124 \\ \cline{2-3}
& (3,5,2,2,2,2) & 304 \\ \cline{2-3}
& (3,3,3,2,2,2) & 240 \\ \cline{2-3}
& (3,3,2,2,2,3) & 150 \\ \cline{2-3}
& (3,4,2,2,2,3) & 1392 \\ \cline{2-3}
& (3,3,2,2,2,4) & 464 \\ \cline{2-3}
& (3,3,2,2,2,5) & 17,040 \\ \hline
Knotted $K_4$ & (3,3,2,2,2,2) & unknown \\ \hline
\end{tabular}
\end{center}
\caption{Size of finite $N$-quandles of exceptional graphs. Labelings that do not correspond to an orbifold of type 4 are marked with an asterisk.}
\label{T:exceptional}
\end{table}

The only quandle in Table \ref{T:exceptional} we were unable to compute was the $(3,3,2,2,2,2)$-quandle for the knotted $K_4$. It is unclear whether this is just due to insufficient computational resources, or whether it may be a counterexample to our Main Conjecture. 

\section{Families of Graphs}\label{S:families}

Now we turn to the graphs which are the singular locus for a spherical orbifold of type 2 (in Dunbar's classification), which fiber over a 2-orbifold.  Dunbar further divides these into types 2a and 2b.  The orbifolds of type 2a are fibered over a 2-orbifold with no boundary; in all these cases, the singular locus is a link.  The orbifolds of type 2b are fibered over a 2-orbifold with boundary components; in these cases, the singular locus usually involves one or more rational tangles. The rational tangle may have a {\em strut} in the innermost twist (corresponding to an exceptional fiber), turning the link into a graph. Figure \ref{F:linktable} shows the links containing rational tangles which are the singular locus for a spherical orbifold; if the rational tangle $p/q$ has $\gcd(p,q) > 1$, then the singular locus is a spatial graph with a strut with label $\gcd(p,q)$.  It is convenient to let the fraction $\frac{0}{q}$ represent the empty rational tangle (just two horizontal arcs) along with a vertical strut labeled $q$. In this section, we are going to consider this special case for two families of links from Figure \ref{F:linktable}. 

\begin{figure}[htbp]
\begin{center}
{$
\begin{array}{cc}
\includegraphics[width=1.5in,trim=0 0 0 0,clip]{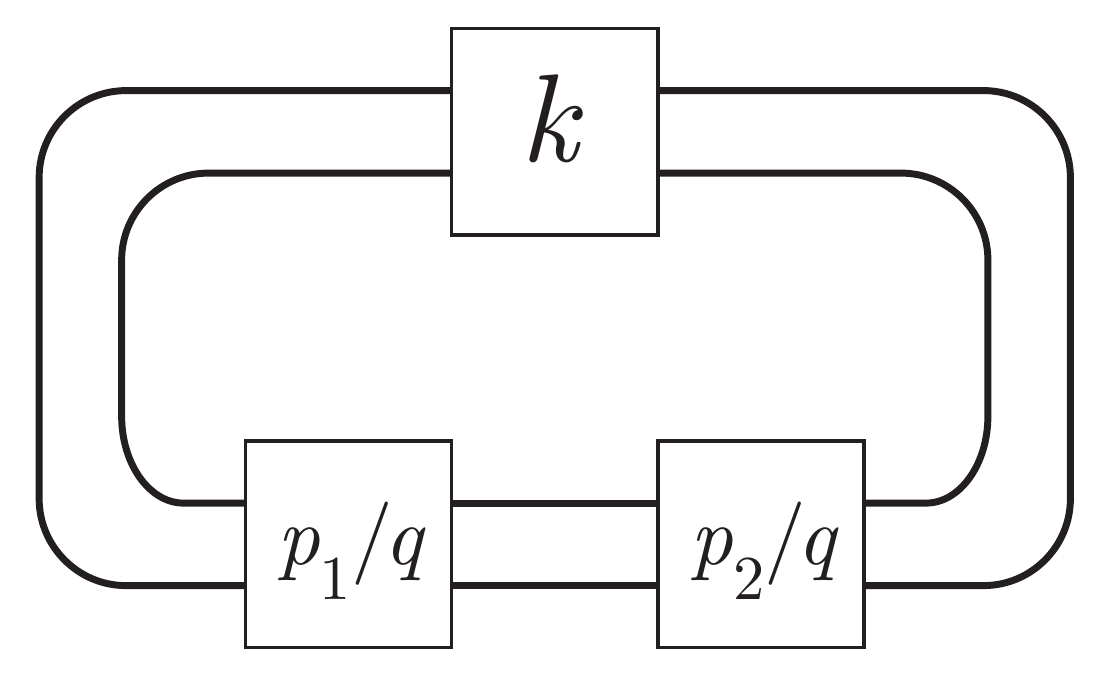}  & \includegraphics[width=1.25in,trim=0 5pt 0 0,clip]{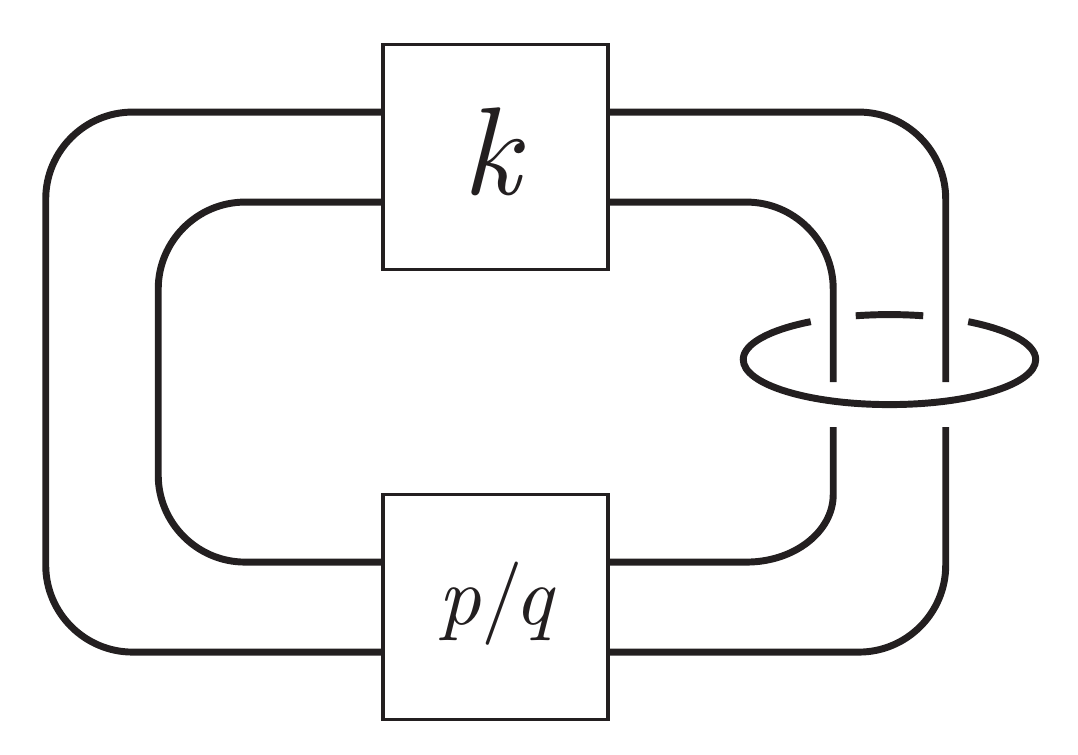} \\
\scriptstyle k+p_1/q+p_2/q \neq 0  & \\
\\
 \includegraphics[width=1.85in,trim=0 0 0 0,clip]{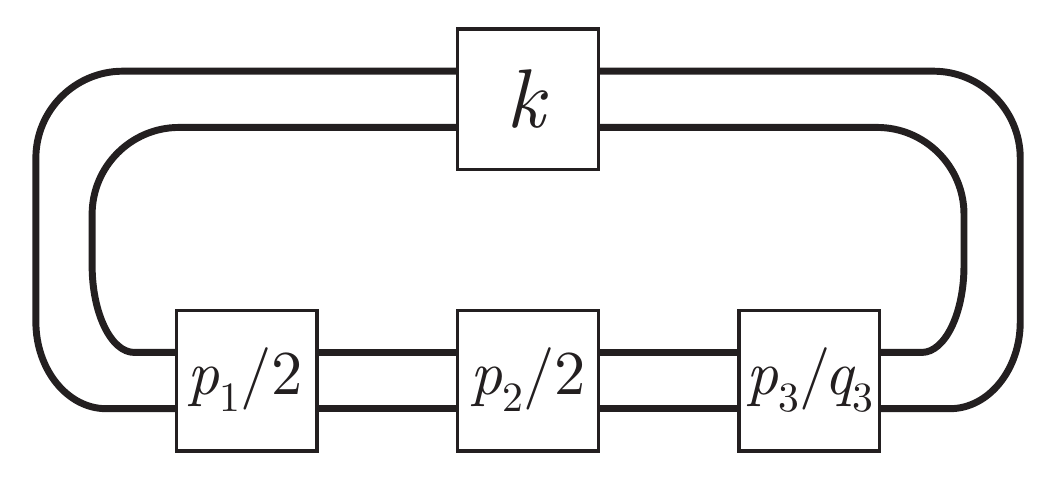} & \includegraphics[width=1.85in,trim=0pt 0pt 0pt 0pt,clip]{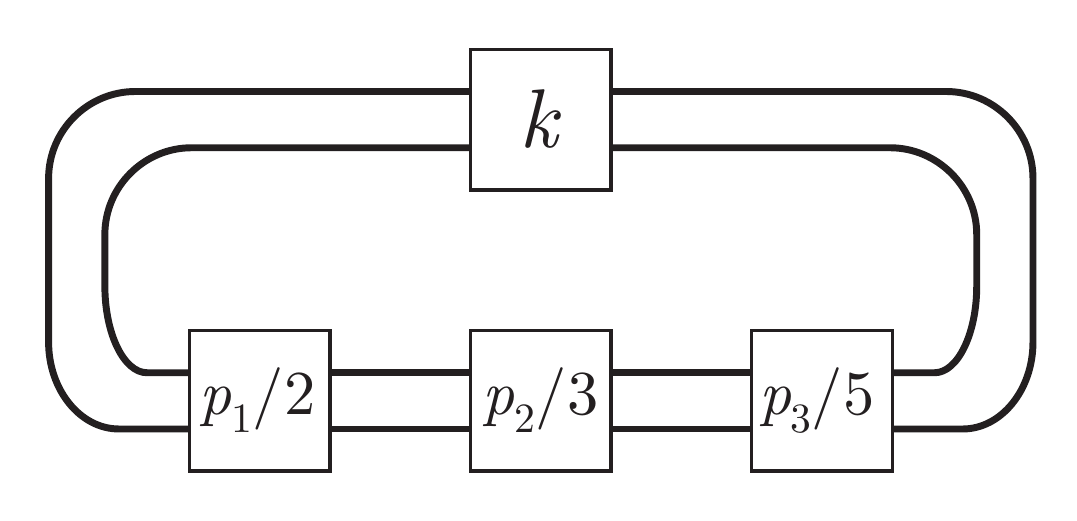}\\
\scriptstyle k+p_1/2+p_2/2+p_3/q_3 \neq 0 & \scriptstyle k+p_1/2+p_2/3+p_3/5 \neq 0\\
\\
\includegraphics[width=1.85in,trim=0pt 0pt 0pt 0pt,clip]{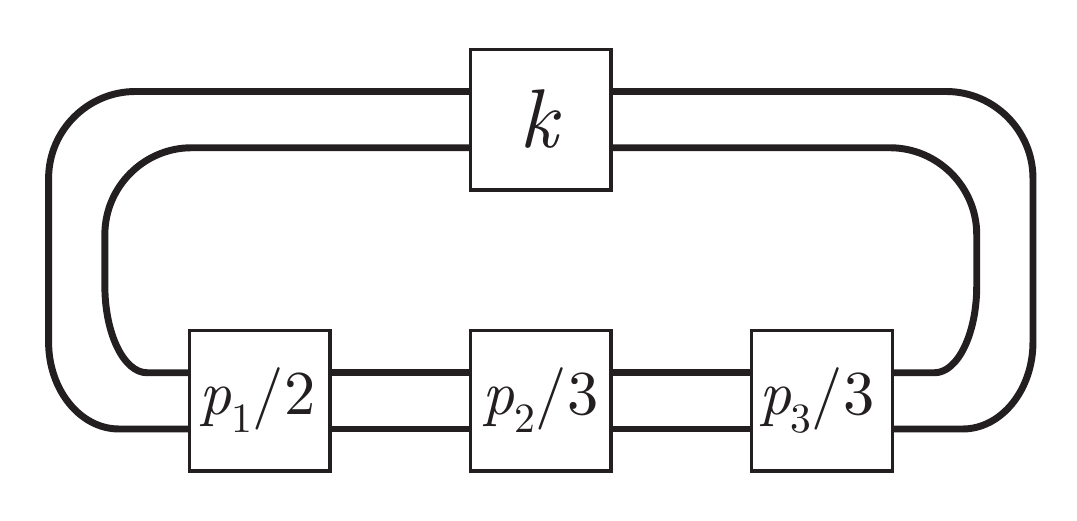}  & \includegraphics[width=1.85in,trim=0pt 0pt 0pt 0pt,clip]{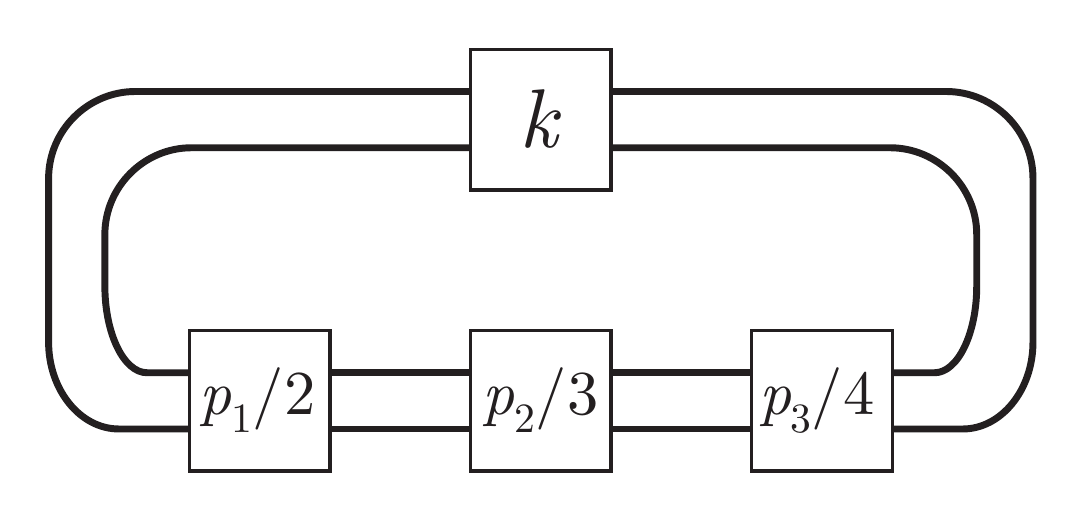}  \\
\scriptstyle k+p_1/2+p_2/3+p_3/3 \neq 0  & \scriptstyle k+p_1/2+p_2/3+p_3/4 \neq 0  
\end{array}
$}
\caption{\parbox{4in}{Links $L \in \mathbb{S}^3$ with finite $Q_n(L)$ which contain rational tangles. \newline Here \fbox{$k$} represents $k$ right-handed half-twists, and \fbox{$p/q$} represents a rational tangle. $L$ is a graph if $\gcd(p,q) > 1$.}}
\end{center}
\label{F:linktable}
\end{figure}


We will consider the first diagram (in the upper left) of Figure \ref{F:linktable}, in the special case when $p_1 = p_2 = 0$ and $q = 1$.  This is the family of graphs where the rational tangles are simply struts labeled $m$ and $n$, as in Figure \ref{F:family1}. We will denote this graph by $G(k, m, n)$. If $m$ or $n$ is 1, we can use Lemma \ref{L:delete} to ignore that strut (the case when they are both 1, giving a twist link, was considered in \cite{CHMS}), giving a graph we denote $G(k, m)$; in this case the underlying graph is either a $\theta$-graph (if $k$ is odd) or a handcuff graph (if $k$ is even).  If $n$ and $m$ are both greater than $1$, then the underlying graph is either $K_4$ (if $k$ is odd) or a double handcuff graph (if $k$ is even).

\begin{figure}[htbp]
$$\includegraphics[width=2in]{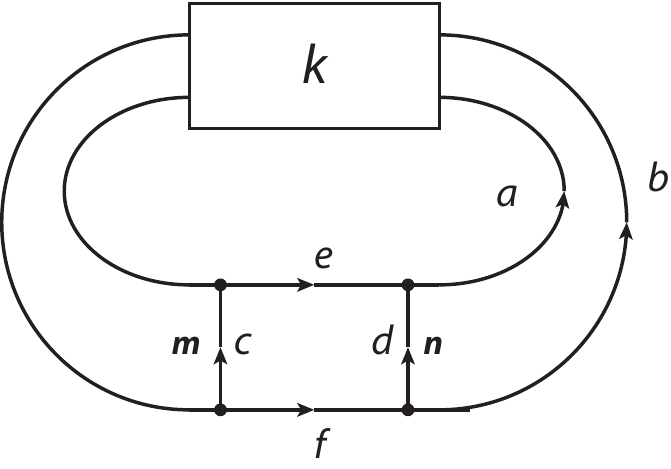} \qquad \rule[-.1in]{0in}{1in}\includegraphics[width=2in]{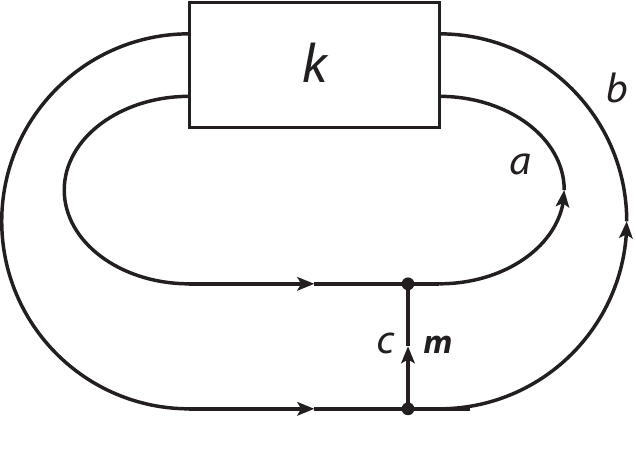}$$
\caption{The labeled graphs $G(k, m, n)$ and $G(k, m)$.}
\label{F:family1}
\end{figure}

Our main result in this section is the following:

\begin{theorem} \label{T:Gkmn}
Let $N = (2, 2, m, n, 2, 2)$ and $G = G(k,m,n)$ (labeled as in Figure \ref{F:family1}).  Then $\vert Q_N(G)\vert = 4kmn + 2km + 2kn$.
\end{theorem}

As a corollary, we will show:

\begin{corollary} \label{C:Gkm}
Let $N = (2, 2, m)$ and $G = G(k,m)$ (labeled as in Figure \ref{F:family1}).  Then $\vert Q_N(G)\vert = 2km + 2k$.
\end{corollary}

Our first task is to find a presentation for $Q_N(G)$, for $N = (2, 2, m, n, 2, 2)$ and $G = G(k,m,n)$. As described in section \ref{S:Cayley}, the Wirtinger presentation would have a generator for every arc of the diagram, and relations for each crossing, vertex and generator. However, we can simplify the presentation by observing that all the generators corresponding to arcs in the block of $k$ half-twists can be written in terms of the generators $a$ and $b$ (see Figure \ref{F:family1}), using the crossing relations. So it's enough to trace these strands through the block of half-twists, to find the labels for the arcs on the left-hand side of the block. These labels are easily determined by an inductive argument, as observed in \cite{Me}.

\begin{lemma}\cite{Me} \label{L:halftwists}
The arcs on either side of the block of $k$ right-handed half-twists are labeled as shown in Figure \ref{F:halftwists} (for $k$ even and $k$ odd).  Here $X = (ba)^t$ and $Y = (ba)X = (ba)^{t+1}$. (If $k < 0$, there are $\vert k \vert$ left-handed half-twists; the same formulas hold, where $(ba)^{-1} = \overline{ba} = ab$.)
\begin{figure}[htbp]
\centerline{\scalebox{1}{\includegraphics{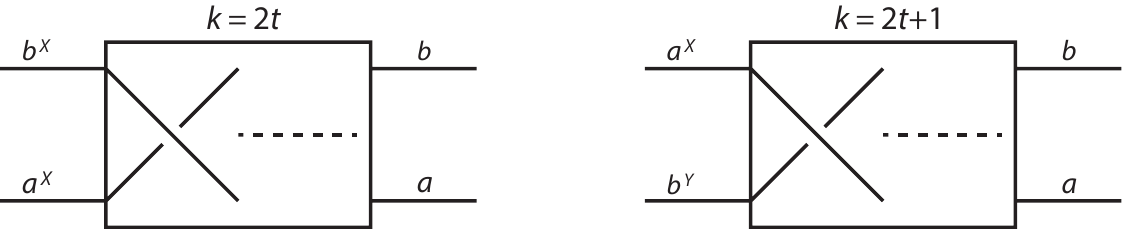}}}
\caption{Blocks of right-handed half-twists.}
\label{F:halftwists}
\end{figure}
\end{lemma}

So we now have a presentation with six generators, and ten relations (four relations for the vertices, and six for the generators). The relations for the generators are (where $x$ is an arbitrary element of the quandle):
$$x^{a^2} = x^{b^2} = x^{e^2} = x^{f^2} = x^{c^m} = x^{d^n} = x.$$
In particular, this means $x^a = x^{\bar{a}}$, $x^b = x^{\bar{b}}$, $x^e = x^{\bar{e}}$ and $x^f = x^{\bar{f}}$ for any element $x$.

The relations for the two right-hand vertices are $x^{de\bar{a}} = x^{dea} = x$ and $x^{bd\bar{f}} = x^{bdf} = x$. For the left hand vertices, we first consider the case when $k = 2t$ is even.  Then, using Lemma \ref{L:halftwists}, we have:
\begin{align*}
(1)\ x &= x^{ca^X \bar{e}} = x^{ca^X e} = x^{ca^{(ba)^t}e} = x^{c(ab)^ta(ba)^t e} = x^{c(ab)^k ae}. \\
(2)\ x &= x^{fc \overline{b^X}} = x^{fc\overline{b^{(ba)^t}}} = x^{fc\overline{(ab)^tb(ba)^t}} = x^{fc \overline{(ab)^{k-1}a}} \\
&= x^{fc \bar{a} (\bar{b}\bar{a})^{k-1}} = x^{fc a (ba)^{k-1}} = x^{fc (ab)^{k-1} a}.
\end{align*}
Now we consider the case when $k = 2t+1$ is odd.  Again using Lemma \ref{L:halftwists}, we have:
\begin{align*}
(1)\ x &= x^{cb^Y \bar{e}} = x^{cb^Y e} = x^{cb^{(ba)^{t+1}}e} = x^{c(ab)^{t+1}b(ba)^{t+1} e} = x^{c(ab)^k ae}. \\
(2)\ x &= x^{fc \overline{a^X}} = x^{fc\overline{a^{(ba)^t}}} = x^{fc\overline{(ab)^ta(ba)^t}} = x^{fc \overline{(ab)^{k-1}a}} \\
&= x^{fc \bar{a} (\bar{b}\bar{a})^{k-1}} = x^{fc a (ba)^{k-1}} = x^{fc (ab)^{k-1} a}.
\end{align*}
So, in fact, we get the same relations (in terms of $k$) in both cases.  The presentation for $Q_N(G)$ is then:
\begin{multline*}
Q_N(G) = \langle a, b, c, d, e, f \mid x^{a^2} = x^{b^2} = x^{e^2} = x^{f^2} = x^{c^m} = x^{d^n} = x, \\
x^{dea} = x, x^{bdf} = x, x^{c(ab)^k ae} = x, x^{fc (ab)^{k-1} a} = x \rangle.
\end{multline*}

\subsection{Relations in $Q_N(G)$} \label{SSS:relations}

In this section, we will prove some useful relations in the quandle $Q_N(G)$. We keep in mind the following observation: if $x^{wy} = x$ for every element $x$ in the quandle, where $w$ is a word in the quandle and $y$ is an element of the quandle, then $(x^y)^{wy} = x^y$, so $x^{yw} = x^{y\inv{y}} = x$.  In other words, if we have a relation $x^w = x$, then we can cyclically permute the factors of $w$ to get more such relations.

\begin{lemma} \label{L:rel1}
For any element $x\in Q_N(G)$,
\begin{enumerate}[label=(\arabic*)]
	\item $x^a = x^{dad} = x^{\inv{d}a\inv{d}} = x^{de} = x^{e\inv{d}}$.
	\item $x^b = x^{dbd} = x^{\inv{d}b\inv{d}} = x^{df} = x^{f\inv{d}}$.
	\item $x^{ab} = x^{dab\inv{d}} = x^{\inv{d}abd} = x^{ef}$.
\end{enumerate}
\end{lemma}
\begin{proof}
Since $x^{dea} = x$, we immediately have $x^a = x^{de}$.  By cyclic permutation, $x^{ade} = x$, so $x^a = x^{e\inv{d}}$.  Then $x^{dad} = x^{de\inv{d}d} = x^{de} = x^a$, and $x^{\inv{d}a\inv{d}} = x^{\inv{d}de\inv{d}} = x^{e\inv{d}} = x^a$.  This gives relation (1).  Similarly, using $x^{bdf} = x$ gives relation (2).

Then $x^{ab} = x^{(dad)(\inv{d}b\inv{d})} = x^{dab\inv{d}}$ and $x^{ab} = x^{(\inv{d}a\inv{d})(dbd)} = x^{\inv{d}abd}$.  Since $x^{\inv{d}a} = x^{(\inv{d}a\inv{d})d} = x^{(e\inv{d})d} = x^e$ and $x^{bd} = x^{\inv{d}(dbd)} = x^{\inv{d}(df)} = x^f$, we conclude that $x^{ab} = x^{ef}$, completing relation (3).
\end{proof}

\begin{lemma} \label{L:rel2}
For any element $x\in Q_N(G)$,
\begin{enumerate}[label=(\arabic*)]
	\item $x^{(ab)^k} = x^{\inv{c}\inv{d}} = x^{(ef)^k}$.
	\item $x^{cd\inv{c}\inv{d}} = x$.
	\item $x^{c^iu} = x^{u\inv{c}^i}$ and $x^{d^ju} = x^{u\inv{d}^j}$ for $u \in \{a, b, e, f\}$, and $i, j \in \Z$.
\end{enumerate}
\end{lemma}
\begin{proof}
We begin with the relation $x^{c(ab)^k ae} = x$. By a cyclic permutation, this means $x^{(ab)^k aec} = x$, so $x^{(ab)^k} = x^{\inv{c}ea}$.  By Lemma \ref{L:rel1}(1), $x^{\inv{c}ea} = x^{\inv{c}e(e\inv{d})} = x^{\inv{c}\inv{d}}$.  So $x^{(ab)^k} = x^{\inv{c}\inv{d}}$.  And since $x^{ab} = x^{ef}$ by Lemma \ref{L:rel1}(3), $x^{(ab)^k} = x^{(ef)^k}$, completing the proof of relation (1).

For relation (2), observe that $x^{cd\inv{c}\inv{d}} = x^{cd(ab)^k}$ by part (1). $x^{cd(ab)^k} = x^{cd(ab)^k\inv{d}d} = x^{c(dab\inv{d})^k d} = x^{c(ab)^k d}$ by Lemma \ref{L:rel1}(3).  Finally, $x^{c(ab)^k d} = x^{c\inv{c}\inv{d}d} = x$ by part (1) again.  Hence $x^{cd\inv{c}\inv{d}} = x$.

From Lemma \ref{L:rel1}, we know that $x^{du} = x^{u\inv{d}}$ and $x^{\inv{d}u} = x^{ud}$ for $u \in \{a, b, e, f\}$.  So $x^{d^ju} = x^{u\inv{d}^j}$ for any integer $j$. It remains to show the same relations using $c$ instead of $d$. To do this, we use the relations $x^{c(ab)^k ae} = x$ and $x^{fc (ab)^{k-1} a} = x$.  These imply that $x^c = x^{ea(ba)^k} = x^{fa(ba)^{k-1}}$ and hence $x^{\inv{c}} = x^{(ab)^kae} = x^{(ab)^{k-1}af}$. So (using Lemma \ref{L:rel1} as needed):
\begin{align*}
x^{c^ia} &= x^{[ea(ba)^k]^ia} =x^{[ea(dba\inv{d})^k]^ia} = x^{[ead(ba)^k\inv{d}]^ia} = x^{[ee(ba)^k\inv{d}]^ia} \\
&= x^{[aa(ba)^k\inv{d}]^ia} = x^{a[a(ba)^k\inv{d}a]^i} = x^{a[a(ba)^ke]^i} = x^{a[(ab)^kae]^i} = x^{a\inv{c}^i}. \\
x^{c^ib} &= x^{[fa(ba)^{k-1}]^i b} = x^{[bda(ba)^{k-1}]^i b} = x^{b [da(ba)^{k-1} b]^i} = x^{b [d(ab)^k]^i} \\
&= x^{b [d\inv{d}(ab)^k d]^i} = x^{b [(ab)^k bf]^i} = x^{b [(ab)^{k-1}af]^i} = x^{b \inv{c}^i}. \\
x^{c^ie} &= x^{[ea(ba)^k]^ie} = x^{e[a(ba)^ke]^i} = x^{e[(ab)^kae]^i} = x^{e\inv{c}^i}. \\
x^{c^if} &= x^{[fa(ba)^{k-1}]^if} = x^{f[a(ba)^{k-1}f]^i} = x^{f[(ab)^{k-1}af]^i} = x^{f\inv{c}^i}.
\end{align*}
This completes the proof of relation (3).
\end{proof}

\subsection{The component $Q_a$ of $Q_N(G)$} \label{SSS:Qa}

The quandle $Q_N(G)$ has six components, one for each of the generators $a, b, c, d, e, f$.  We let $Q_u$ denote the component containing the generator $u$, for $u = a, b, c, d, e, f$. We will begin by describing $Q_a$. We will show that $\abs{Q_a} = kmn$; the Cayley graph for $Q_a$ can be viewed as having $k$ ``layers," each of which contains $mn$ vertices. Each layer can be embedded as an $m\times n$ grid on a torus. The edges labeled $c$ and $d$ connect vertices within each layer, while the edges labeled $a, b, e, f$ connect vertices in adjacent layers. Figure~\ref{F:a_component} shows the case when $m = n = 3$ and $k = 4$.

\begin{figure}[h]
    \centering
    \includegraphics[width=0.8\textwidth]{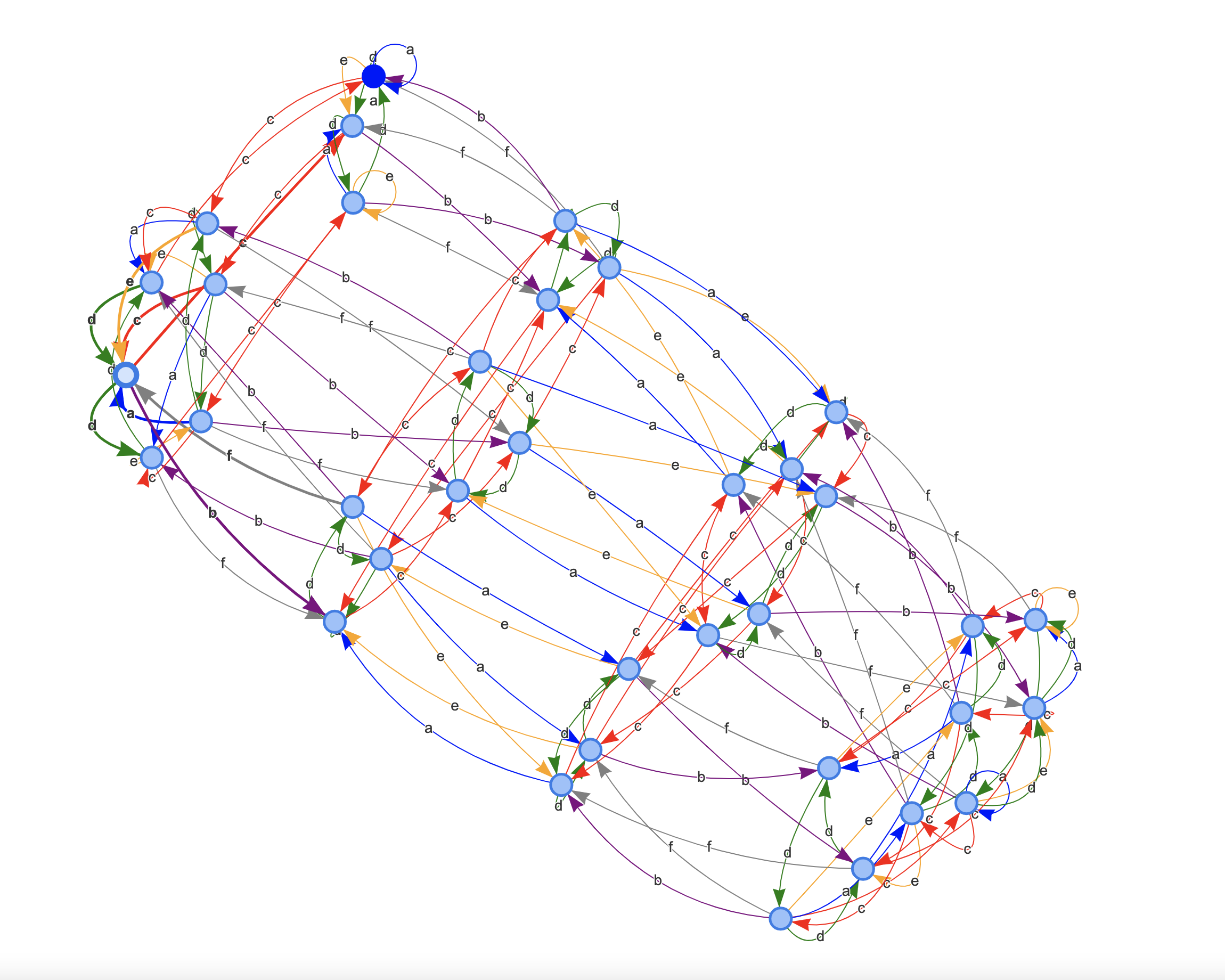}
    \caption{Cayley graph of $Q_a$ when $m = n =3$ and $k=4$.}
    \label{F:a_component}
\end{figure}

We will denote the elements of $Q_a$ by $x_{p, q, r}$, where $p, q, r \in \Z$. We let $x_{0,0,0} = a$, and define:
$$x_{p,q,r} = \left\{\begin{matrix} a^{(ba)^t c^q d^r}, &\text{ if }p = 2t \\ a^{(ba)^tb c^q d^r}, &\text{ if }p = 2t+1 \end{matrix} \right.$$
Observe that $x_{p, q+m, r} = x_{p,q,r}$ and $x_{p, q, r+n} = x_{p,q,r}$, so we may assume $0 \leq q \leq m-1$ and $0 \leq r \leq n-1$ (in other words, we interpret these subscripts modulo $m$ and $n$, respectively). To show that these are all the vertices in the Cayley graph, we will show that the action of each of the generators on $x_{p,q,r}$ gives another element $x_{p',q',r'}$.

We first consider the action of $d$ and $c$.  Clearly, $x_{p,q,r}^d = a^{(wc^qd^r) d} = a^{wc^qd^{r+1}} = x_{p,q,r+1}$ (where $w = (ba)^t$ or $(ba)^tb$).  Also, since $x^{dc} = x^{cd}$ for any $x$ (by Lemma \ref{L:rel2}(2)), $x_{p,q,r}^c = a^{(wc^qd^r)c} = a^{wc^{q+1}d^r} = x_{p,q+1,r}$. Similarly, $x_{p,q,r}^{\inv{d}} = x_{p,q,r-1}$ and $x_{p,q,r}^{\inv{c}} = x_{p,q-1,r}$.

Since $a, b, e, f$ all have order 2, the action of each generator and its inverse are the same. However, the action does depend on whether $p$ is odd or even.
\begin{align*}
x_{p,q,r}^a &= \left\{\begin{matrix} a^{(ba)^t c^q d^r a}, &\text{ if }p = 2t \\ a^{(ba)^tb c^q d^r a}, &\text{ if }p = 2t+1 \end{matrix} \right. \\
&= \left\{\begin{matrix} a^{(ba)^t a \inv{c}^q \inv{d}^r}, &\text{ if }p = 2t \\ a^{(ba)^tb a \inv{c}^q \inv{d}^r}, &\text{ if }p = 2t+1 \end{matrix} \right. \text{ (by Lemma \ref{L:rel2}(3))}\\
&= \left\{\begin{matrix} a^{(ba)^{t-1}b c^{-q} d^{-r}}, &\text{ if }p = 2t \\ a^{(ba)^{t+1} c^{-q} d^{-r}}, &\text{ if }p = 2t+1 \end{matrix} \right. \\
&= \left\{\begin{matrix} x_{p-1,-q,-r}, &\text{ if }p = 2t \\ x_{p+1, -q, -r}, &\text{ if }p = 2t+1 \end{matrix} \right. 
\end{align*}
\begin{align*}
x_{p,q,r}^b &= \left\{\begin{matrix} a^{(ba)^t c^q d^r b}, &\text{ if }p = 2t \\ a^{(ba)^tb c^q d^r b}, &\text{ if }p = 2t+1 \end{matrix} \right. \\
&= \left\{\begin{matrix} a^{(ba)^t b \inv{c}^q \inv{d}^r}, &\text{ if }p = 2t \\ a^{(ba)^tb b \inv{c}^q \inv{d}^r}, &\text{ if }p = 2t+1 \end{matrix} \right. \text{ (by Lemma \ref{L:rel2}(3))}\\
&= \left\{\begin{matrix} a^{(ba)^t b c^{-q} d^{-r}}, &\text{ if }p = 2t \\ a^{(ba)^t c^{-q} d^{-r}}, &\text{ if }p = 2t+1 \end{matrix} \right. \\
&= \left\{\begin{matrix} x_{p+1,-q,-r}, &\text{ if }p = 2t \\ x_{p-1, -q, -r}, &\text{ if }p = 2t+1 \end{matrix} \right. \\
x_{p,q,r}^e &= x_{p,q,r}^{ad} = \left\{\begin{matrix} x_{p-1,-q,-r+1}, &\text{ if }p = 2t \\ x_{p+1, -q, -r+1}, &\text{ if }p = 2t+1 \end{matrix} \right. \\
x_{p,q,r}^f &= x_{p,q,r}^{bd} = \left\{\begin{matrix} x_{p+1,-q,-r+1}, &\text{ if }p = 2t \\ x_{p-1, -q, -r+1}, &\text{ if }p = 2t+1 \end{matrix} \right.
\end{align*}

We've observed that we may assume $0 \leq q \leq m-1$ and $0 \leq r \leq n-1$. The following lemma shows that we may assume $0 \leq p \leq k-1$. 

\begin{lemma} \label{L:ends}
For any integers $q$ and $r$, \begin{enumerate}[label=(\arabic*)]
	\item $x_{-1,q,r} = x_{0,q,r}$,
	\item if $k$ is even, then $x_{k,q,r} = x_{k-1,q+1,r+1}$, and 
	\item if $k$ is odd, then $x_{k,q,r} = x_{k-1,q-1,r-1}$.
\end{enumerate}
\end{lemma}
\begin{proof}
To prove (1), observe that
$$x_{-1,q,r} = x_{2(-1)+1,q,r} = a^{(ba)^{-1}b c^q d^r} = a^{abb c^q d^r} = a^{c^qd^r} = x_{0,q,r}$$
If $k = 2t$ is even, then (using Lemma \ref{L:rel2}(1)):
\begin{align*}
x_{k, q, r} &= a^{(ba)^t c^q d^r} = a^{(ba)^t [(ab)^k dc] c^q d^r} = a^{(ab)^t c^{q+1} d^{r+1}} \\
&= a^{(ba)^{t-1}b c^{q+1} d^{r+1}} = x_{k-1, q+1, r+1}.
\end{align*}
Similarly, if $k = 2t+1$ is odd, then:
\begin{align*}
x_{k, q, r} &= a^{(ba)^t b c^q d^r} = a^{(ba)^t b [(ba)^k \inv{c}\inv{d}] c^q d^r} = a^{(ab)^t a c^{q-1} d^{r-1}} \\
&= a^{(ba)^t c^{q-1} d^{r-1}} = x_{k-1, q-1, r-1}.
\end{align*}
\end{proof}
Since the actions of the generators $a, b, e, f$ only increment $p$ by $\pm 1$, starting from $a = x_{0,0,0}$, we can't get values of $p$ less than 0 or greater than $k-1$.  So we may assume $0 \leq p \leq k-1$.

Since $0 \leq p \leq k-1$, $0 \leq q \leq m-1$ and $0 \leq r \leq n-1$, we have $kmn$ vertices $x_{p,q,r}$. Now we need to check that the quandle relations are satisfied at every vertex with no further collapsing. It is easy to check from the actions described above that:
$$x_{p,q,r}^{a^2} = x_{p,q,r}^{b^2} = x_{p,q,r}^{c^m} = x_{p,q,r}^{d^n} = x_{p,q,r}^{e^2} = x_{p,q,r}^{f^2} = x.$$
Now we will check the next two relations in $Q_N(G)$:
\begin{align*}
x_{p,q,r}^{dea} &= x_{p,q,r+1}^{ea} = \left\{\begin{matrix} x_{p-1,-q,-r}^a, &\text{ if }p = 2t \\ x_{p+1, -q, -r}^a, &\text{ if }p = 2t+1 \end{matrix} \right. \\
&= \left\{\begin{matrix} x_{(p-1)+1,-(-q),-(-r)}, &\text{ if }p = 2t \\ x_{(p+1)-1, -(-q), -(-r)}, &\text{ if }p = 2t+1 \end{matrix} \right. \\
&= x_{p,q,r}
\end{align*}
\begin{align*}
x_{p,q,r}^{bdf} &= \left\{\begin{matrix} x_{p+1,-q,-r}^{df}, &\text{ if }p = 2t \\ x_{p-1, -q, -r}^{df}, &\text{ if }p = 2t+1 \end{matrix} \right. \\
&= \left\{\begin{matrix} x_{p+1,-q,-r+1}^f, &\text{ if }p = 2t \\ x_{p-1, -q, -r+1}^f, &\text{ if }p = 2t+1 \end{matrix} \right. \\
&= \left\{\begin{matrix} x_{(p+1)-1,-(-q),-(-r+1)+1}, &\text{ if }p = 2t \\ x_{(p-1)+1, -(-q), -(-r+1)+1}, &\text{ if }p = 2t+1 \end{matrix} \right. \\
&= x_{p,q,r}.
\end{align*}
For the final two relations, it will be convenient to consider the cases when $p$ and $k$ are even or odd separately. We also observe that
$$x_{p,q,r}^{ab} = \left\{\begin{matrix} x_{p-2,q,r}, &\text{ if }p = 2t \\ x_{p+2,q,r}, &\text{ if }p = 2t+1 \end{matrix} \right.$$
$$x_{p,q,r}^{ba} = \left\{\begin{matrix} x_{p+2,q,r}, &\text{ if }p = 2t \\ x_{p-2,q,r}, &\text{ if }p = 2t+1 \end{matrix} \right.$$
We first consider the case when $p = 2t$ and $k = 2s$, with $0 \leq p \leq k-1$.  Then by tracing the action of the generators we compute:
\begin{align*}
x_{p,q,r}^{c(ab)^k ae} &= x_{p,q+1,r}^{(ab)^k ae} \\
&= x_{p,q+1,r}^{(ab)^t a (ba)^s b (ab)^{s-t-1} ae} \\
&= x_{0, q+1, r}^{a (ba)^s b (ab)^{s-t-1} ae} = x_{-1,-q-1,-r}^{(ba)^s b (ab)^{s-t-1} ae} \\
&= x_{0, -q-1, -r}^{(ba)^s b (ab)^{s-t-1} ae} \text{ (by Lemma \ref{L:ends}(1))}\\
&= x_{k, -q-1, -r}^{b (ab)^{s-t-1} ae} = x_{k-1, -q, -r+1}^{b (ab)^{s-t-1} ae} \text{ (by Lemma \ref{L:ends}(2))} \\
&= x_{k-2,q,r-1}^{(ab)^{s-t-1} ae} = x_{(k-2)-2(s-t-1), q, r-1}^{ae} = x_{p, q, r-1}^{ae} \\
&= x_{p-1, -q, -r+1}^e = x_{p, q,r}
\end{align*}
and
\begin{align*}
x_{p,q,r}^{fc (ab)^{k-1} a} &= x_{p+1,-q,-r+1}^{c (ab)^{k-1} a} = x_{p+1,-q+1,-r+1}^{(ab)^{k-1} a} \\
&= x_{p+1,-q+1,-r+1}^{(ab)^{s-t-1} a (ba)^s b (ab)^{t-1} a} \\
&= x_{(p+1)+(k-p-2),-q+1,-r+1}^{a (ba)^s b (ab)^{t-1} a} = x_{k-1,-q+1,-r+1}^{a (ba)^s b (ab)^{t-1} a} \\
&= x_{k,q-1,r-1}^{(ba)^s b (ab)^{t-1} a} = x_{k-1, q,r}^{(ba)^s b (ab)^{t-1} a} \text{ (by Lemma \ref{L:ends}(2))} \\
&= x_{-1, q,r}^{b (ab)^{t-1} a} = x_{0,q,r}^{b (ab)^{t-1} a} \text{ (by Lemma \ref{L:ends}(1))}\\
&= x_{1,-q,-r}^{(ab)^{t-1} a} = x_{1+(p-2), -q,-r}^a = x_{p-1, -q,-r}^a \\
&= x_{p, q, r}.
\end{align*}
The proofs for the other combinations of the parities of $p$ and $k$ are similar. So there is no further collapsing, and the elements of $Q_a$ are exactly the elements $x_{p,q,r}$ for $0 \leq p \leq k-1$, $0 \leq q \leq m-1$ and $0 \leq r \leq n-1$.  So $\abs{Q_a} = kmn$.

\subsection{The component $Q_d$ of $Q_N(G)$} \label{SSS:Qd}

Now we will describe the Cayley graph for the component $Q_d$ of $Q_N(G)$, and prove that it has $2km$ elements.  Figure~\ref{F:d_component} shows the Cayley graph for $Q_d$ in the case when $m = n = 3$ and $k = 4$. In general, the Cayley graph of $Q_d$ consists of $2k$ $m$-cycles arranged in a loop. The $m$-cycles are made up of edges labeled $c$, while adjoining cycles in the loop are connected by edges labeled $a, b, e, f$. The edges labeled $d$ are small loops at each vertex of the Cayley graph.

\begin{figure}[ht]
    \centering
    \includegraphics[width=0.7\textwidth]{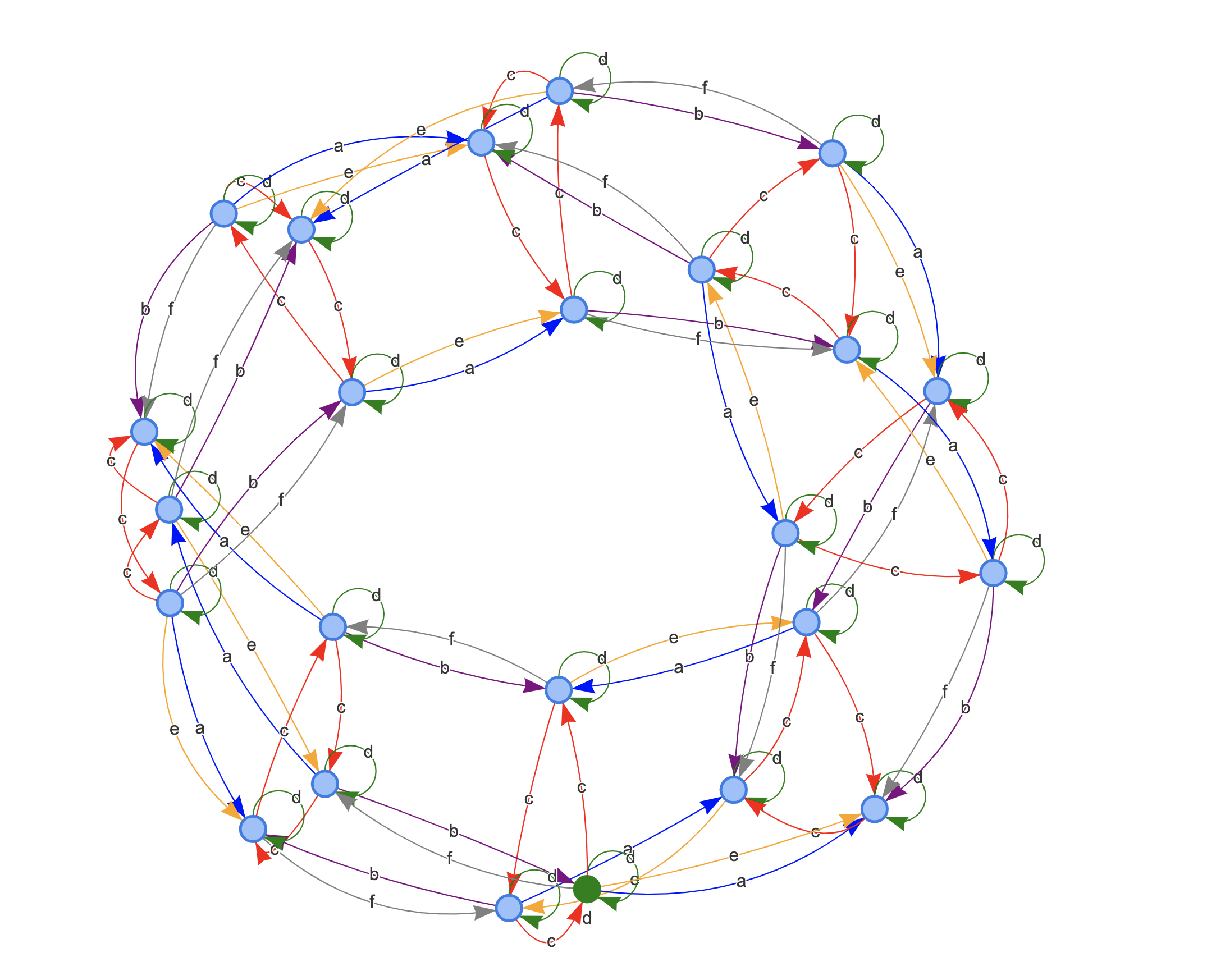}
    \caption{Cayley graph of $Q_d$ when $m=3$, $n=3$ and $k=4$.}
    \label{F:d_component}
\end{figure}

We will denote the elements of $Q_d$ by $y_{p,q}$, where $y_{0,0} = d$ and (for $p, q \in \Z$)
$$y_{p,q} = \left\{\begin{matrix} d^{(ab)^t c^q}, &\text{ if }p = 2t \\ d^{(ab)^ta c^q}, &\text{ if }p = 2t+1 \end{matrix} \right.$$
Since $y_{p, q+m} = y_{p,q}$, we may assume $0 \leq q \leq m-1$. Now we need to determine the action of each generator on $y_{p,q}$. The action of $c$ is easy to see, moving around the $m$-cycle:
$$y_{p,q}^c = d^{wc^q c} = d^{wc^{q+1}} = y_{p,q+1}.$$
The action of $d$ takes every element of $Q_d$ back to itself, giving the loops in Figure \ref{F:d_component}:
\begin{align*}
y_{p,q}^d &= \left\{\begin{matrix} d^{(ab)^t c^q d}, &\text{ if }p = 2t \\ d^{(ab)^ta c^q d}, &\text{ if }p = 2t+1 \end{matrix} \right. \\
&= \left\{\begin{matrix} d^{(ab)^t d c^q}, &\text{ if }p = 2t \\ d^{(ab)^ta d c^q}, &\text{ if }p = 2t+1 \end{matrix} \right. \text{ (by Lemma \ref{L:rel2}(2))} \\
&= \left\{\begin{matrix} d^{d (ab)^t c^q}, &\text{ if }p = 2t \\ d^{\inv{d} (ab)^ta c^q}, &\text{ if }p = 2t+1 \end{matrix} \right. \text{ (by Lemma \ref{L:rel1})}\\
&= \left\{\begin{matrix} d^{(ab)^t c^q}, &\text{ if }p = 2t \\ d^{(ab)^ta c^q}, &\text{ if }p = 2t+1 \end{matrix} \right. \\
&= y_{p,q}.
\end{align*}
The actions of $a, b, e, f$ move between the $m$-cycles:
\begin{align*}
y_{p,q}^a &= \left\{\begin{matrix} d^{(ab)^t c^q a}, &\text{ if }p = 2t \\ d^{(ab)^ta c^q a}, &\text{ if }p = 2t+1 \end{matrix} \right. \\
&= \left\{\begin{matrix} d^{(ab)^t a \inv{c}^q}, &\text{ if }p = 2t \\ d^{(ab)^ta a \inv{c}^q}, &\text{ if }p = 2t+1 \end{matrix} \right. \text{ (by Lemma \ref{L:rel2}(3))}\\
&= \left\{\begin{matrix} d^{(ab)^t a \inv{c}^q}, &\text{ if }p = 2t \\ d^{(ab)^t \inv{c}^q}, &\text{ if }p = 2t+1 \end{matrix} \right. \\
&= \left\{\begin{matrix} y_{p+1,-q}, &\text{ if }p = 2t \\ y_{p-1,-q}, &\text{ if }p = 2t+1 \end{matrix} \right.
\end{align*}
\begin{align*}
y_{p,q}^b &= \left\{\begin{matrix} d^{(ab)^t c^q b}, &\text{ if }p = 2t \\ d^{(ab)^ta c^q b}, &\text{ if }p = 2t+1 \end{matrix} \right. \\
&= \left\{\begin{matrix} d^{(ab)^t b \inv{c}^q}, &\text{ if }p = 2t \\ d^{(ab)^ta b \inv{c}^q}, &\text{ if }p = 2t+1 \end{matrix} \right. \text{ (by Lemma \ref{L:rel2}(3))}\\
&= \left\{\begin{matrix} d^{(ab)^{t-1} a \inv{c}^q}, &\text{ if }p = 2t \\ d^{(ab)^{t+1} \inv{c}^q}, &\text{ if }p = 2t+1 \end{matrix} \right. \\
&= \left\{\begin{matrix} y_{p-1,-q}, &\text{ if }p = 2t \\ y_{p+1,-q}, &\text{ if }p = 2t+1 \end{matrix} \right.
\end{align*}
To determine the action of $e$ and $f$, we use the relations $x^{dea} = x$ and $x^{bdf} = x$.  Since $y_{p,q}^d = y_{p,q}$, we find that $e$ and $f$ have the same actions as $a$ and $b$.
$$y_{p,q}^e = y_{p,q}^{ad} = y_{p,q}^a$$
$$y_{p,q}^f = y_{p,q}^{bd} = y_{p,q}^b$$
To put bounds on $p$, observe that 
\begin{align*}
y_{p,q}^{ab} &= \left\{\begin{matrix} y_{p+1,-q}^b, &\text{ if }p = 2t \\ y_{p-1,-q}^b, &\text{ if }p = 2t+1 \end{matrix} \right. \\
&= \left\{\begin{matrix} y_{p+2,q}, &\text{ if }p = 2t \\ y_{p-2,q}, &\text{ if }p = 2t+1 \end{matrix} \right.
\end{align*}
In particular, this means that 
$$y_{p,q}^{(ab)^k} = \left\{\begin{matrix} y_{p+2k,q}, &\text{ if }p = 2t \\ y_{p-2k,q}, &\text{ if }p = 2t+1 \end{matrix} \right.$$
But, by Lemma \ref{L:rel2}(1), $y_{p,q}^{(ab)^k} = y_{p,q}^{\inv{c}\inv{d}} = y_{p,q-1}^{\inv{d}} = y_{p,q-1}$.  Therefore, 
$$y_{p+2k, q} = \left\{\begin{matrix} y_{p, q-1}, &\text{ if }p = 2t \\ y_{p, q+1}, &\text{ if }p = 2t+1 \end{matrix} \right.$$
Hence, we may assume that $0 \leq p \leq 2k-1$.

It is straightforward to check that all the relations now hold at every vertex of $Q_d$, so the Cayley graph is complete, and $\abs{Q_d} = 2km$.

\subsection{The components $Q_b$, $Q_c$, $Q_e$ and $Q_f$ of $Q_N(G)$} \label{SSS:Qbcef}

The components $Q_b$, $Q_e$ and $Q_f$, like $Q_a$, have $kmn$ elements, while $Q_c$ has $2kn$ elements.  The result for $Q_b$, $Q_e$ and $Q_f$ can be proved by arguments similar to those in section \ref{SSS:Qa}; however, here we will give a more topological argument. Consider the isotopy shown in Figure \ref{F:twist}, where we perform a flype on the bottom portion of $G(k,m,n)$. Since the two new crossings have opposite signs, the number of positive half-twists is still $k$. So this isotopy induces an automorphism of $Q_N(G)$ that interchanges $a$ and $b$, interchanges $e$ and $f$, and fixes $c$ and $d$, but in the Cayley graph reverses the orientation of the edges labeled $c$ and $d$. Hence the Cayley graphs for $Q_a$ and $Q_b$ are isomorphic, as are the Cayley graphs for $Q_e$ and $Q_f$.

\begin{figure}[htbp]
$$\includegraphics[width=3.5in]{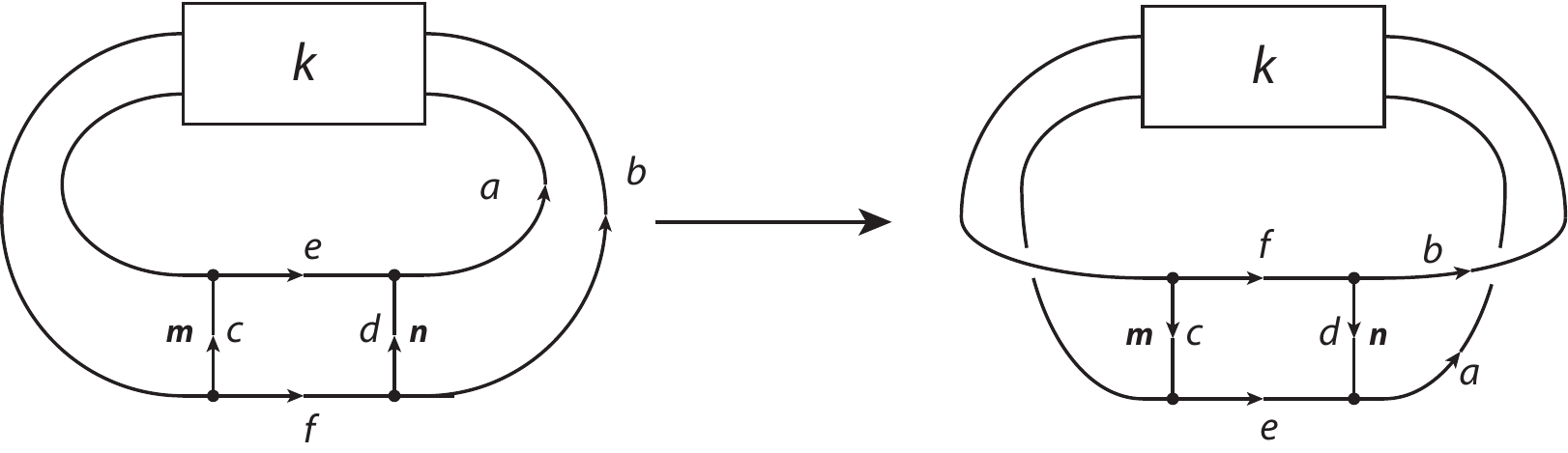}$$
\caption{Performing a flype on $G(k,m,n)$.}
\label{F:twist}
\end{figure}

We also consider the isotopy shown in Figure \ref{F:slide}.  Here we first slide the edge labeled $c$ through the block of $k$ half-twists (if $k$ is even, the orientation is the same afterwards; if $k$ is odd it is reversed), and then rotate the graph $180^\circ$ around a vertical axis. The induced automorphism on $Q_N(G)$ interchanges $a$ and $e$, interchanges $b$ and $f$, and fixes $c$ and $d$ (though it may reverse the orientation of the edges labeled $c$ in the Cayley graph). So the Cayley graphs for $Q_a$ and $Q_e$ are isomorphic, as are the Cayley graphs for $Q_b$ and $Q_f$.

\begin{figure}[htbp]
$$\includegraphics[width=5in]{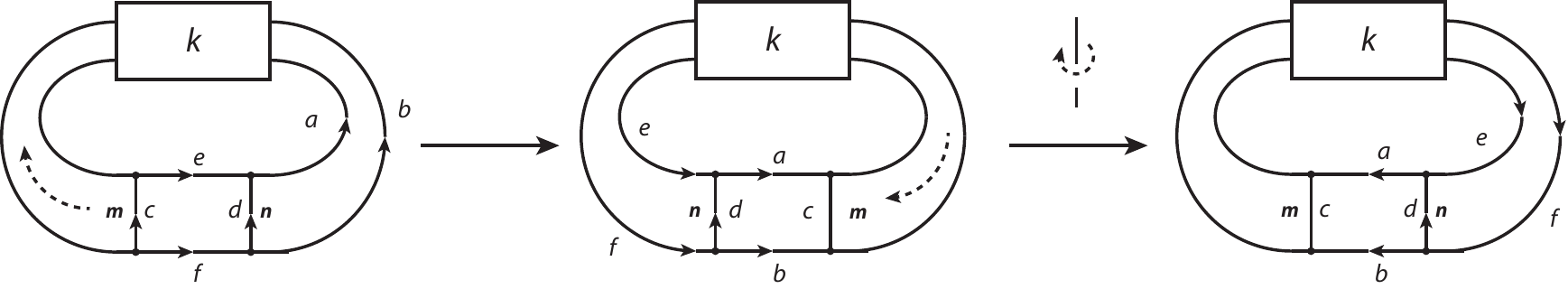}$$
\caption{Performing a flype on $G(k,m,n)$.}
\label{F:slide}
\end{figure}

We conclude that $Q_a$, $Q_b$, $Q_e$ and $Q_f$ all have isomorphic Cayley graphs, and hence all have $kmn$ elements. The Cayley graph for $Q_c$ can be computed similarly to that for $Q_d$ (as is clear, in particular, from the middle diagram in Figure \ref{F:slide}), so $\abs{Q_c} = 2kn$.

Combining all of these results gives us $\abs{Q_N(G)} = 4kmn + 2km + 2kn$, proving Theorem \ref{T:Gkmn}.  As a corollary, we consider the quandle $Q_{(2,2,m)}(G(k,m))$.  In this case, by Lemmas \ref{L:delete} and \ref{L:vertex}, we let $n = 1$, and delete the components corresponding to the generators $d, e, f$. This gives us $\vert Q_{(2,2,m)}(G(k,m))\vert = 2km + 2k$, proving Corollary \ref{C:Gkm}.

\section{Open questions} \label{S:questions}

There are many open questions that can be investigated further. In section \ref{SS:properties}, we investigated how a few very simple graph operations affected the $N$-quandle; it is natural to ask how other graph operations affect the fundamental quandle.

\begin{question}
How do graph operations such as edge deletion, edge contraction, etc., affect the $N$-quandle of the graph? How are the $N$-quandles of a minor of a spatial graph related to the $N$-quandle of the larger graph?
\end{question}

In this paper, we only considered one direction of the Main Conjecture, showing that the spatial graphs appearing in Dunbar's classification of orbifolds have finite $N$-quandles.  There is still much work to be done here, beginning with the potential counterexample.

\begin{question}
Is $Q_{(3,3,2,2,2,2)}(\text{knotted }K_4)$ finite?
\end{question}

We can also investigate the families of links and graphs in Figure \ref{F:linktable}.  This will require dealing with the general rational tangles, as was done in \cite{HS1} and \cite{Me}, but with the added complexity of a strut inserted into the tangle.

\begin{question}
Do the families of graphs in Figure \ref{F:linktable} all have finite $N$-quandles?
\end{question}

And, of course, this still leaves open the other direction of the Main Conjecture:

\begin{question}
Are there other graphs which have finite $N$-quandles, which do not satisfy the criterion of the Main Conjecture?
\end{question}

This seems like a much harder problem, since the proof for $n$-quandles of links in \cite{HS2} relies on constructions such as branched covering spaces that do not easily extend to graphs.

\bibliography{Nquandle.bib}

\begin{thebibliography}{10}

\bibitem{CHMS}
A.~Crans, J.~Hoste, B.~Mellor, and P.~D. Shanahan.
\newblock Finite $n$-quandles of torus and two-bridge links.
\newblock {\em Journal of Knot Theory and Its Ramifications}, 28, 2019.

\bibitem{DU}
W.~Dunbar.
\newblock Geometric orbifolds.
\newblock {\em Rev. Mat. Univ. Complut. Madrid}, 1:67--99, 1988.

\bibitem{FR}
R.~Fenn and C.~Rourke.
\newblock Racks and links in codimension two.
\newblock {\em Journal of Knot Theory and Its Ramifications}, 1:343--406, 1992.

\bibitem{HS1}
J.~Hoste and P.~D. Shanahan.
\newblock Involutory quandles of $(2,2,r)$-{M}ontesinos links.
\newblock {\em Journal of Knot Theory and Its Ramifications}, 26, 2017.

\bibitem{HS2}
J.~Hoste and P.~D. Shanahan.
\newblock Links with finite $n$-quandles.
\newblock {\em Algebraic and Geometric Topology}, 17:2807--2823, 2017.

\bibitem{HS3}
J.~Hoste and P.~D. Shanahan.
\newblock An enumeration process for racks.
\newblock {\em Math. of Computation}, 88:1427--1448, 2019.

\bibitem{JO2}
D.~Joyce.
\newblock An algebraic approach to symmetry with applications to knot theory.
\newblock Ph.D. thesis, University of Pennsylvania, 1979.

\bibitem{JO}
D.~Joyce.
\newblock A classifying invariant of knots, the knot quandle.
\newblock {\em Journal of Pure and Applied Algebra}, 23:37--65, 1982.

\bibitem{Ma}
S.~V. Matveev.
\newblock Distributive groupoids in knot theory.
\newblock {\em Math. USSR Sbornik}, 47:73--83, 1984.

\bibitem{Me2}
B.~Mellor.
\newblock {C}ayley graphs for finite {$N$}-quandles.
\newblock \url{http://blakemellor.lmu.build/research/Nquandle/index.html},
  2020.

\bibitem{Me}
B.~Mellor.
\newblock Finite involutory quandles of two-bridge links with an axis.
\newblock {\em Journal of Knot Theory and Its Ramifications}, 31(2), 2022.

\bibitem{MS}
B.~Mellor and R.~Smith.
\newblock {$N$}-quandles of links.
\newblock {\em Topology and its Applications}, 294, 2021.

\bibitem{Ni}
M.~Niebrzydowski.
\newblock Coloring invariants of spatial graphs.
\newblock {\em Journal of Knot Theory and Its Ramifications}, 19(6):829--841,
  2010.

\bibitem{TC}
J.~Todd and H.~S.~M. Coxeter.
\newblock A practical method for enumerating cosets of a finite abstract group.
\newblock {\em Proceedings of the Edinburgh Mathematical Society, Series II},
  5:26--34, 1936.

\bibitem{WI}
S.~Winker.
\newblock Quandles, knot invariants, and the $n$-fold branched cover.
\newblock Ph.D. thesis, University of Illinois, Chicago, 1984.

\end{thebibliography}

\end{document}